\documentclass{article}
\usepackage{amsmath,amsfonts,amssymb,amsthm}
\usepackage{color}

\providecommand{\R}{\mathbb{R}}
\providecommand{\C}{\mathbb{C}}

\providecommand{\om}{\omega}

\renewcommand{\leq}{\leqslant}

\renewcommand{\Re}{\mbox{Re}}

\renewcommand{\div}{\operatorname{div}}
\newcommand{\curl}{\operatorname{curl}}

\newcommand{\Id}{\operatorname{Id}}
\newtheorem{Theorem}{Theorem}
\newtheorem{Definition}{Definition}

\newtheorem{Proposition}{Proposition}
\newtheorem{Lemma}{Lemma}

\begin{document}

\author{Olivier Glass\footnote{CEREMADE, UMR CNRS 7534,
Universit\'e Paris-Dauphine, 
Place du Mar\'echal de Lattre de Tassigny,
75775 Paris Cedex 16,
FRANCE
},
Franck Sueur\footnote{CNRS, UMR 7598, Laboratoire Jacques-Louis Lions, F-75005, Paris, France}
\footnote{UPMC Univ Paris 06, UMR 7598, Laboratoire Jacques-Louis Lions, F-75005, Paris, France}}

\date{\today}
\title{Low regularity solutions for the two-dimensional \\ ``rigid body  + incompressible Euler"  system.  }
\maketitle
\begin{abstract}
In this note we consider the motion of a solid body in a two dimensional  incompressible perfect fluid. 
We prove the global existence of solutions in the case where the initial vorticity belongs to $L^p$ with $p>1$ and is compactly supported. 
We do not assume that the energy is finite.
\end{abstract}
\section{Introduction}
\label{Intro}
In this paper, we consider the motion of a solid body in a perfect incompressible fluid which fills the plane, in a low-regularity setting. Let us describe the system under view.
Let $\mathcal{S}_0$ be a closed, bounded, connected and simply connected subset of the plane with smooth boundary.
We assume that the body initially occupies the domain $\mathcal{S}_0$ and rigidly moves so that at  time $t$  it occupies an isometric  domain denoted by $\mathcal{S}(t)$.
We denote $\mathcal{F} (t) := \R^2  \setminus \mathcal{S}(t) $ the domain occupied by the fluid  at  time $t$ starting from the initial domain $\mathcal{F}_{0}  := \R^2 \setminus {\mathcal{S}}_{0} $. \par
The equations modelling the dynamics of the system then read
\begin{eqnarray}
\displaystyle \frac{\partial u }{\partial t}+(u  \cdot\nabla)u   + \nabla p =0 && \text{for} \ x\in \mathcal{F} (t), \label{Euler1}\\
\div u   = 0 && \text{for} \ x\in \mathcal{F}(t) , \label{Euler2} \\
u  \cdot n =   u_\mathcal{S} \cdot n && \text{for}  \  x\in \partial \mathcal{S}  (t),   \label{Euler3} \\
\lim_{|x|\to \infty} |u| =0,& & \\
m h'' (t) &=&  \int_{\partial  \mathcal{S} (t)} p n \, ds ,  \label{Solide1} \\ 
\mathcal{J} r' (t) &= &   \int_{\partial  \mathcal{S} (t)} (x-  h (t) )^\perp \cdot pn   \, ds , \label{Solide2} \\
u |_{t= 0} = u_0 & &  \text{for}  \  x\in  \mathcal{F}_0 ,  \label{Eulerci2} \\
h (0)= h_0 , \ h' (0)=  \ell_0 , & &   r  (0)=  r _0.  \label{Solideci}
\end{eqnarray}
Here $u=(u_1,u_2)$ and $p$ denote the velocity and pressure fields,
$m$ and $ \mathcal{J}$ denote respectively the mass and the moment of inertia of the body  while the fluid  is supposed to be  homogeneous of density $1$, to simplify the notations.
When $x=(x_1,x_2)$ the notation $x^\perp $ stands for $x^\perp =( -x_2 , x_1 )$, 
$n$ denotes  the unit normal vector pointing outside the fluid,  $h'(t)$
is the velocity of the center of mass  $h (t)$ of the body and $r(t)$ denotes the angular velocity of the rigid body. Finally we denote by $u_{{\mathcal S}}$ the velocity of the body:
\begin{equation} \label{VeloBody}
u_\mathcal{S} (t,x) =   h' (t)+ r (t) (x-  h (t))^\perp .
\end{equation}
\ \par
Without loss of generality and for the rest of the paper, we assume from now on that 
\begin{equation*}
h_{0}= 0,
\end{equation*}
which means that the body is centered at the origin at the initial time $t=0$. \par
\ \par
The equations \eqref{Euler1} and \eqref{Euler2} are the incompressible Euler equations, the condition \eqref{Euler3} means that the boundary is impermeable and the equations \eqref{Solide1} and \eqref{Solide2} are the Newton's balance law for linear and angular momenta. \par 
\ \par
The main goal of this paper is to prove the global existence of some solutions this system in the case where 
\begin{gather}
\label{deco}
u_0  \in    \tilde{L}^{2} := L^{2} ({\mathcal F}_0) \oplus \R  H_{0}  ,
\\ \label{TourbillonYudo}
w_0 := \curl u_0  \in  L^{p}_{c } ({\mathcal F}_0), \text{ with } p> 1 .
\end{gather}
Above 
\begin{equation*}
H_{0} (x) := \frac{x^{\perp}}{| x |^{2} } \ \text{ for } \ x \in {\mathcal F}_{0},
\end{equation*}
and the notation $L_c^{p}({\mathcal F}_0)$ denotes the subset of the functions in $L^{p}({\mathcal F}_0)$ which are compactly supported in the closure of ${\mathcal F}_0$. Note in particular that $u_{0}$ is not necessarily of finite energy. \par
\ \par
Our result below extends the result \cite{WangXin} of Wang and Xin where it is assumed that the vorticity $w_0$ belongs to $ L^1 \cap L^{p}$ with  $p \in (\frac{4}{3},+\infty ]$ (but is not necessarily compactly supported) and that the solution has finite energy; it also extends the result of \cite{GLS} where it is assumed that  the vorticity $w_0$  belongs to $L^{p}$ with $p \in (2,+\infty ]$ and is compactly supported, with possibly infinite kinetic energy. Also, this provides a counterpart of the results by DiPerna and Majda \cite{DiPernaMajda} about the global  existence of weak solutions in the case of the fluid alone when the vorticity belongs in $ L^1 \cap L^{p}$ with $p >1$.
Let us also mention that the global existence and uniqueness of finite energy classical solutions to the problem \eqref{Euler1}--\eqref{Solideci} has been tackled by Ortega, Rosier and Takahashi in \cite{ort1}. 
This result was extended in \cite{GS}  to the case of infinite energy. \par
\ \par
\section{Main result}
\label{MR}
Before giving the main statement of this paper, let us give our definition of a weak solution. This relies on a change of variable which we now describe. \par
Since $\mathcal{S}(t)$ is obtained from ${\mathcal S}_{0}$ by a rigid motion, there exists a rotation matrix 
\begin{eqnarray*} 
Q (t):= 
\begin{bmatrix}
\cos  \theta (t) & - \sin \theta (t) \\
\sin  \theta (t) & \cos  \theta (t)
\end{bmatrix},
\end{eqnarray*}
such that  the position $\eta (t,x) \in \mathcal{S} (t)$  at 
the time $t$ of the point fixed to the body with an initial position $x$ is 
\begin{eqnarray*} 
\eta (t,x) := h (t) + Q (t)x .
\end{eqnarray*}
The angle $\theta$ satisfies
\begin{equation*}
\theta'(t) = r (t),
\end{equation*}
and we choose $\theta (t)$ such that $\theta (0) =  0$.
\par

In order to transfer the equations in the body frame we apply the following isometric change of variable:
\begin{equation*} 
\left\{
\begin{array}{l}
 v (t,x)=Q (t)^T\,  u(t,Q(t)x+h(t)), \\
 q (t,x)=p(t,Q(t)x+h(t)), \\
 {\ell} (t)=Q (t)^T \ h' (t) .
\end{array}\right.
\end{equation*}
so that the equations  (\ref{Euler1})-(\ref{Solideci})  become
\begin{eqnarray}
\label{Euler11}
\displaystyle \frac{\partial v}{\partial t}
+ \left[(v-\ell-r x^\perp)\cdot\nabla\right]v 
+ r v^\perp +\nabla q =0 && x\in \mathcal{F}_{0} ,\\
\label{Euler12}
\div v = 0 && x\in \mathcal{F}_{0} , \\
\label{Euler13}
v\cdot n = \left(\ell +r x^\perp\right)\cdot n && x\in \partial \mathcal{S}_0, \\
\label{Solide11}
m \ell'(t)=\int_{\partial \mathcal{S}_0} q n \, ds-mr (\ell)^\perp, & & \\
\label{Solide12}
\mathcal{J} r'(t)=\int_{\partial \mathcal{S}_0} x^\perp \cdot q n \, ds, & &  \\
\label{Euler1ci}
v(0,x)= v_0 (x) && x\in \mathcal{F}_{0} ,\\
\label{Solide1ci}
\ell(0)= \ell_0,\ r (0)= r_0 . 
\end{eqnarray}
We also define
\begin{eqnarray*} 
\omega(t,x) := w(t, Q(t)x+h(t) )  = \curl v(t,x).
\end{eqnarray*}
Taking the curl of the equation \eqref{Euler11} we get
\begin{equation}
\label{vorty}
\partial_t  \omega + \left[(v-\ell-r x^\perp)\cdot\nabla\right]  \omega =0 \text{ for }
x \in \mathcal{F}_{0} .
\end{equation}
Let us now give a global weak formulation of the problem by considering ---for the solution as well as for test functions--- a velocity field on the whole plane, with the constraint to be rigid on $ \mathcal{S}_{0} $. \par
We introduce the following space
\begin{equation*}
\mathcal{H}: = \left\{\Psi\in L^2_{loc} (\mathbb{R}^2) ;  \  \div \Psi =0 \ \text{ in } \ \mathbb{R}^2 \ \text{ and } \
D \Psi =0 \ \text{ in } \ \mathcal{S}_0 \right\}, \ \text{ where } \ D\Psi := \nabla \Psi + \nabla \Psi^{T} .
\end{equation*}
It is classical that 
\begin{equation*}
\mathcal{H} = \left\{ \Psi\in L^2_{loc}  (\mathbb{R}^2) ;  \  \div \Psi =0 \ \text{ in } \ \mathbb{R}^2 \ \text{ and } \
\exists (\ell_\Psi , r_\Psi ) \in \mathbb{R}^2 \times \mathbb{R},  \  \forall x \in  \mathcal{S}_0 , \
\Psi  (x) = \ell_\Psi +  r_\Psi x^\perp \right\} , 
\end{equation*}
and the $(\ell_\Psi , r_\Psi )$ above are unique.

Let us also introduce $\tilde{\mathcal{H} }$  the set of the test functions $\Psi$ in $\mathcal{H} $  whose restriction $\Psi |_{ \overline{\mathcal{F}_{0}}}$ to the closure of the fluid domain belongs to $C^1_c ( \overline{\mathcal{F}_{0}})$, and, for  $T>0$, 
$\tilde{\mathcal{H} }_{T}$ the set of the test functions $\Psi$ in $C^1([0,T]; \mathcal{H} )$ whose restriction $\Psi |_{[0,T] \times  \overline{\mathcal{F}_{0}}}$ to the closure of the fluid domain belongs to $C^1_c ( [0,T] \times \overline{\mathcal{F}_{0}})$. \par
When $(\overline{u},\overline{v}) \in \mathcal{H} \times \tilde{\mathcal{H} }$, we denote by 
\begin{equation*}
(\overline{u},\overline{v})_\rho
:= \int_{\mathbb{R}^2} (\rho \xi_{\mathcal{S}_{0} } + \xi_{\mathcal{F}_{0} }) \, \overline{u} \cdot \overline{v} 
= m \, \ell_u \cdot \ell_v + \mathcal{J} r_u r_v +  \int_{\mathcal{F}_{0} } u \cdot v  ,
\end{equation*}
where the notation $\xi_{A} $ stands for the characteristic function of the set $A$, $ u  \in  L^2_{loc}(\mathcal{F}_{0})$ denotes  the restrictions of $\overline{u}$ to $\mathcal{F}_{0}$ and $ \rho$ denotes the density in the solid body. Note that this density is related to $m$ and ${\mathcal J}$ by
\begin{equation*}
m = \int_{{\mathcal S}_{0}} \rho(x) \, dx
\ \text{ and } \ 
{\mathcal J} = \int_{{\mathcal S}_{0}} \rho(x) |x|^{2} \, dx.
\end{equation*}
Our definition of a weak solution is the following.
\begin{Definition}[Weak Solution] \label{DefWS}
Let us be given  $ \overline{v}_0 \in \mathcal{H}$ and $T>0$.
We say that  $ \overline{v} \in C ([0,T]; \mathcal{H}-w)$ is a weak solution of \eqref{Euler11}--\eqref{Solide1ci} in $[0,T]$ if for any test function $\Psi \in \tilde{\mathcal{H} }_{T}$,
\begin{equation*}
(  \Psi(T,\cdot ) , \overline{v}(T,\cdot) )_\rho
- (  \Psi (0,\cdot ) , \overline{v}_0 )_\rho
= \int_0^T \left( \frac{\partial \Psi}{\partial t},\overline{v} \right)_\rho \, dt
+ \int_0^T \!\! \int_{\mathcal{F}_0} v\cdot\left(\left(v-\ell_v  - r_v x^\perp  \right)\cdot\nabla\right)\Psi \, dx\, dt
- \int_0^T \!\! \int_{\mathcal{F}_0} r v^\perp   \cdot \Psi \, dx \, dt
\end{equation*}
We will equivalently say that $(\ell,r,u)$ is a weak solution of \eqref{Euler11}--\eqref{Solide1ci}.
\end{Definition}
Definition \ref{DefWS} is legitimate since a classical solution of \eqref{Euler11}--\eqref{Solide1ci} in $[0,T]$ is also a weak solution. This follows easily from an integration by parts in space which provides
\begin{equation} \label{SemiW}
(\partial_{t} \overline{v} , \Psi  )_\rho 
= \int_{\mathcal{F}_0} v \cdot\left(\left(v-\ell_v - r_v x^\perp \right)\cdot\nabla\right)\Psi \, dx
- \int_{\mathcal{F}_0} r v^\perp   \cdot \Psi  \, dx ,
\end{equation}
and then from an  integration by parts in time. \par
\ \par
One has the following result of existence of weak solutions for the above system, the initial position of the solid being given. 
\begin{Theorem} \label{ThmYudo}
Let $p > 1$  and $T>0$.
For any $v_0 \in  \tilde{L}^{2}$, $(\ell_0 , r_0) \in \R^2 \times \R$, such that:
\begin{equation} \label{CondCompatibilite}
\div v_0 =0 \text{ in } {\mathcal F}_0 \ \text{ and } \  v_0   \cdot  n = (\ell_0 + r_0 x^{\perp})   \cdot  n \text{ on } \partial \mathcal{S}_0,
\end{equation}
and $\omega_0 := \curl v_0 $ satisfies \eqref{TourbillonYudo},
there exists a weak solution of the system such that $\omega \in L^\infty (0,T; (L^1 \cap L^p) ({\mathcal F}_0 ) )$. \par
\end{Theorem}
The rest of the paper is devoted to the proof of Theorem  \ref{ThmYudo}.
\section{Representation of the velocity in the body frame}
\label{Sec:Velocity}
In this section, we study the elliptic $\div$/$\curl$ system which allows to pass from the vorticity to the velocity field, in the body frame. 
We do not claim any originality here, but rather collect some well-known properties which will be useful in the sequel. 
We refer here for instance to \cite{GT,Kikuchi83,MP,ift_lop_euler,GLS}.
\subsection{Green's function and  Biot-Savart operator}
\label{Sec:GreensFunction}
We denote by $G (x,y)$ the Green's function of $\mathcal{F}_{0}$ with Dirichlet boundary conditions. 
We also introduce the function $K (x,y)=\nabla^\perp G (x,y)$ known as the kernel of the Biot-Savart operator  $K [\om]$ which therefore acts on $\om \in L^p_{c}  (\mathcal{F}_{0})$  through the formula 
\begin{equation}
 \label{correct}
K[\om](x)= \int_{\mathcal{F}_{0}}K (x,y) \om(y) \, dy.
\end{equation}
Then $K[\om]$ belongs to $L^2 (\mathcal{F}_{0}) \cap W^{1,p}({\mathcal F}_{0})$, is divergence-free and tangent to the boundary, satisfies $\curl K[\om]=\om$ in $\mathcal{F}_{0}$ and its circulation around $\partial \mathcal{S}_0$ is given by
\begin{equation*}
\int_{\partial \mathcal{S}_0  }   K [ \omega] \cdot {\tau} \, ds = -  \int_{\mathcal{F}_{0}  }  \omega \, dx ,
\end{equation*}
where $\tau$ is the tangent unit vector field on $\partial {\mathcal S}_{0}$. \par
Moreover, there exists $C>0$ (depending only on $\mathcal{F}_{0}$, on $\text{supp }  \om $ and on $p$) such that for any $\om\in L^p_{c} (\mathcal{F}_{0})$, 
\begin{equation} \nonumber
\|  K[\om]  \|_{L^2 (\mathcal{F}_{0})} +  \| \nabla  K[\om]  \|_{L^p (\mathcal{F}_{0})} \leqslant 
C \|  \om  \|_{L^p (\mathcal{F}_{0})}  . 
\end{equation}
\subsection{Harmonic field}
To take the velocity circulation around the body into account, the following vector field will be useful. There exists a unique solution $H$ vanishing at infinity of 
\begin{gather*}
\div H = 0 \quad   \text{for}  \ x\in  \mathcal{F}_{0}, \\
\curl H = 0 \quad   \text{for}  \ x\in  \mathcal{F}_{0}, \\
H \cdot n = 0 \quad   \text{for}  \ x\in   \partial \mathcal{S}_0, \\
\int_{\partial \mathcal{S}_0 } H \cdot \tau \, ds = 1 .
\end{gather*}
This solution is smooth and decays like $1/ | x| $ at infinity.
We also have:
\begin{equation} \label{EstiHache}
\nabla H  \in L^\infty (\mathcal{F}) \ \text{ and } \
 H -  \frac{x^{\perp}}{2 \pi | x |^{2}}, \ \nabla H , \ H^\perp - x^\perp  \cdot\nabla H \in L^2 (\mathcal{F}) .
\end{equation}
Let us stress in particular that the second estimate above yields
\begin{equation*}
\tilde{L}^{2}  =  L^{2} ({\mathcal F}_0) \oplus \R  H .
\end{equation*}
Therefore in the sequel we will rather use this last decomposition than \eqref{deco}.
\subsection{Kirchoff potentials}
\label{KirPo}
Now to lift harmonically the boundary conditions, we will make use of  the  Kirchoff potentials, which are the solutions $\Phi:=(\Phi_{i})_{i=1,2, 3}$ of the following problems: 
\begin{equation} \label{t1.3sec}
-\Delta \Phi_i = 0 \quad   \text{for}  \ x\in \mathcal{F}_{0}   ,
\end{equation}
\begin{equation} \label{t1.4sec}
\Phi_i \longrightarrow 0 \quad  \text{for}  \ x \rightarrow  \infty, 
\end{equation}
\begin{equation} \label{t1.5sec}
\frac{\partial \Phi_i}{\partial n}=K_i 
\quad  \text{for}  \  x\in \partial \mathcal{F}_{0}   ,
\end{equation}
where
\begin{equation} \label{t1.6sec}
(K_{1},\, K_{2}, \, K_{3}) :=(n_1,\, n_2 ,\, x^\perp \cdot n).
\end{equation}
These functions are smooth and decay at infinity as follows:
\begin{equation} \nonumber
\nabla \Phi_i = {\mathcal O}\left( \frac{1}{|x|^{2}}\right) \ \text{ and } \nabla^{2}   \Phi_i  = {\mathcal O}\left( \frac{1}{|x|^{3}}\right) \  \text{ as } x \rightarrow \infty .
\end{equation}
\subsection{Velocity decomposition}
Using the functions defined above, we deduce the following proposition.
\begin{Proposition} \label{sisi}
Let $p>1$. Let us be given $\om\in  L^p_{c} (\mathcal{F}_{0})$, $\ell$ in $\R^2$, $r$ and $\gamma $ in  $\R$.
Then there is a unique solution $v$ in $ \tilde{L}^{2} \cap W^{1,p}_{loc}({\mathcal F}_{0})$ of
\begin{equation*}
\left\{ \begin{array}{l} 
 \div v = 0,  \quad   \text{for}  \ x\in  \mathcal{F}_{0} , \\
 \curl v  = \omega  \quad  \text{for}  \ x\in  \mathcal{F}_{0}  , \\
 v \cdot n = \left(\ell+r x^\perp\right)\cdot n  \quad   \text{for}  \ x\in \partial \mathcal{S}_0  , \\
 v \longrightarrow 0 \quad  \text{as}  \ x \rightarrow  \infty, \\
 \int_{ \partial \mathcal{S}_0} v  \cdot  \tau \, ds=  \gamma .
\end{array} \right.
\end{equation*}
Moreover $v$ is given by 
\begin{equation*}
v = K [\omega ] + (\gamma + \alpha )  H + \ell_1 \nabla \Phi_1 + \ell_2 \nabla \Phi_2
+ r \nabla \Phi_3 , \text{ with } \alpha :=   \int_{\mathcal{F}_{0}  }  \omega \, dx .
\end{equation*}
Finally, there exists $C>0$ (depending only on $\mathcal{F}_{0}$, on $\text{supp }  \om $ and on $p$) such that for any $\om\in L^p_{c} (\mathcal{F}_{0})$, 
\begin{equation} \nonumber
\| \nabla v \|_{L^p (\mathcal{F}_{0})} \leqslant 
C (\|  \om  \|_{L^p (\mathcal{F}_{0})} + |   \ell|  +|  r|   + |  \gamma |  ). 
\end{equation}
\end{Proposition}
\section{Regularization of the initial data}
\label{Sec:Regu}
A general strategy for obtaining a weak solution is to regularize the initial data so that one gets a sequence of initial data which generate some classical solutions, and then to pass  to the limit with respect to the regularization parameter in the weak formulation of the equations.  
We first benefit from the previous section to establish the following approximation result which will be applied to the initial data in the sequel.
Let   $(\ell_{0}, r_{0}, v_{0}) $ as in Theorem \ref{ThmYudo}. 
According to Proposition \ref{sisi} the function  
\begin{equation*}
\tilde{v} _0 := {v}_{0} - \beta H , \text{  where  } \beta := \gamma + \alpha,
\end{equation*}
belongs to $L^{2} ({\mathcal F}_0) $. \par
\begin{Proposition} \label{PropRegu}
There exists $( v^{n}_{0} )_n$ in $\tilde{L}^{2}$, satisfying \eqref{CondCompatibilite} with the same right hand side, such that
\begin{equation*}
\int_{ \partial \mathcal{S}_0} v^{n}_{0}  \cdot  \tau \, ds= \int_{ \partial \mathcal{S}_0} v_{0}  \cdot  \tau \, ds ,
\end{equation*}
\begin{equation*}
( v^{n}_{0} - \beta H )_n \longrightarrow \tilde{v} _0 \ \text{ in } L^{2} ({\mathcal F}_0),
\end{equation*}
and such that
\begin{equation*}
\omega^{n} := \curl v^{n}_{0}  \longrightarrow \omega_{0} \ \text{ in } \ L^{p} ({\mathcal F}_0). 
\end{equation*}

\end{Proposition}
\begin{proof}
%

Let us consider a sequence of smooth mollifiers $(\eta_{n})_{n}$ supported in a fixed compact subset of $\R^{2}$ such that $\int_{{\mathcal F}_0} \eta_{n} \, dx =1$.
Let us denote by $\overline{\omega}_{0}$ the extension  of  $\omega_{0}$ by $0$ in $ \mathcal{S}_0 $. We consider  the functions $\omega^{n}_{0} := (\eta_{n} * \overline{\omega}_{0} )\vert_{\mathcal{F}_0}$ obtained by restriction to $ \mathcal{F}_0$ of the convolution product $\eta_{n} *\overline{\omega}_{0} $. 
These functions $\omega^{n}_{0}$ are smooth, supported in a compact $K$ (not depending on $n$)  of $\overline{{\mathcal F}}_0$ and  $(\omega^{n}_{0} )_{n}$ converges to  $\omega_{0}$ in $L^{p}$. This in particular entails, thanks to the formula \eqref{correct}, that
$$ \int_{ \partial \mathcal{S}_0}  \Big( K [\omega^{n}_{0} ]  - K [\omega_{0} ]  \Big) \cdot  \tau \, ds =
-  \int_{ \mathcal{F}_0}  (\omega_{0}^{n} - \omega_{0} ) dx \rightarrow 0 .$$
Then we define 
\begin{equation*}
v_{0}^{n} = K [\omega_{0}^{n} ] + \beta^{n} H + \ell_1 \nabla \Phi_1 + \ell_2 \nabla \Phi_2
+ r \nabla \Phi_3 , \text{ with }  \beta^{n} :=  \beta +  \int_{\mathcal{F}_0}  (\omega_{0}^{n} - \omega_{0} ) dx ,
\end{equation*}
which is convenient, according to the previous section.
\end{proof}
%

%
%
%
%
%
%
%
%
%
%
\section{A priori estimates}
\label{Sec:APE}
In this section we consider a smooth solution $(\ell ,r ,v )$ of  the problem  (\ref{Euler11})--(\ref{Solide1ci}) with, for any $t \in [0,T] $, $\text{supp } \omega  (t, \cdot)$ lying in a compact of $\mathcal{F}_{0}$.  Our goal is to derive some a priori bounds satisfied by any such solution. \par
We introduce
\begin{equation} \label{allo}
 \tilde{v}: = v - \beta H  .
\end{equation}
Note that due to the support of the vorticity one has
\begin{equation} \label{decay}
 \tilde{v} = {\mathcal O}\left( \frac{1}{|x|^{2}}\right) \ \text{ and } \nabla  \tilde{v} = {\mathcal O}\left( \frac{1}{|x|^{3}}\right) \  \text{ as } x \rightarrow \infty .
\end{equation}
The results of this section will be applied in the sequel to the 
classical solutions generated by the approximation sequence of the previous section.
\subsection{Vorticity}
Due to the equation of vorticity \eqref{vorty} and since we are considering regular solutions,
the generalized enstrophies are conserved as time proceeds, in particular, we have for any $t>0$,
\begin{eqnarray} \label{ConsOmega}
\| \omega  (t, \cdot) \|_{L^{p}(\mathcal{F}_{0})} 
= \| w_{0} \|_{ L^{p}( \mathcal{F}_{0} )},
\quad \| \omega(t,\cdot) \|_{ L^{1}(\mathcal{F}_{0})} 
= \| w_{0} \|_{ L^{1}( \mathcal{F}_{0} )}  .
\end{eqnarray}
One has also the following conservations:
\begin{gather}
\label{Kelvin}
\gamma =  \int_{  \partial \mathcal{S}_0} v  \cdot  \tau \, ds = \int_{  \partial \mathcal{S}(t)} u  \cdot  \tau \, ds
= \int_{  \partial \mathcal{S}_{0}} u_{0}  \cdot  \tau \, ds, \\
\label{DefAlpha2}
\alpha = \int_{ \mathcal{F}_0} \omega(t,x) \, dx = \int_{\mathcal{F}(t)} w(t,x) \, dx = \int_{\mathcal{F}_{0}} w_{0}(x) \, dx.
\end{gather}
\subsection{Energy-like estimate}
\label{NRJ}
Despite the fact that we are not considering finite-energy solutions, we have the following result.
\begin{Proposition} \label{PE2}
There exists a constant $C > 0$ (depending only on $\mathcal{S}_0 $, $m $, $J$, and $\beta$) such that for any smooth solution $(\ell ,r ,v)$ of the problem  \eqref{Euler11}--\eqref{Solide1ci} on the time interval $\lbrack 0, T \rbrack$, the energy-like quantity defined by:
\begin{equation*}
E (t) := \frac{1}{2} \left( m  \ell(t)^2 +   \mathcal{J} r(t)^2 +  \int_{ \mathcal{F}_{0} }  \tilde{v} (t,\cdot )^2 dx \right),
\end{equation*}
satisfies the inequality 
\begin{equation*}
 E (t) \leqslant  E (0) e^{Ct} +  e^{Ct} - 1 .
\end{equation*}
\end{Proposition}
\begin{proof}
We start by recalling that a classical solution satisfies \eqref{SemiW}, and we use the 
decomposition \eqref{allo} in the left hand side and an integration by parts for the first term of the right hand side to get that 
for any test function $\Psi \in \tilde{\mathcal{H} }_{T}$,
\begin{equation*}
m \, \ell' \cdot \ell_\Psi + \mathcal{J} r' r_\Psi +  \int_{\mathcal{F}_{0} } \partial_{t} \tilde{v}  \cdot \Psi 
 = - \int_{\mathcal{F}_0}  \Psi \cdot\left(\left(v-\ell - r x^\perp \right)\cdot\nabla\right) v 
 +  \int_{\mathcal{F}_0} r v^\perp   \cdot \Psi.
\end{equation*}
Then, by a standard regularization process, we observe that the previous identity is still valid for the test function $\Psi $ defined by $ \Psi (t,x) =  \tilde{v}(t,x)  $ for $(t,x)$ in $ [ 0,T] \times \mathcal{F}_0$ and  $ \Psi (t,x) =   \ell(t)  +  r(t) x^\perp $ for  $(t,x)$ in $ [ 0,T] \times \mathcal{S}_0$. Hence we get:
\begin{eqnarray*}
E' (t)
&=& - \int_{\mathcal{F} _0 }  \tilde{v} \cdot\left(\left(v-\ell - r x^\perp \right)\cdot\nabla\right)  {v} \, dx
-  \int_{\mathcal{F}_0} r v^\perp   \cdot  \tilde{v}  \, dx \\ 
&=& - \int_{\mathcal{F} _0 }  \tilde{v}  \cdot \Big( (  {v}  
- \ell-rx^\perp) \cdot\nabla  \tilde{v}   \Big)
- \beta  \int_{\mathcal{F} _0 }  \tilde{v}  \cdot ( \tilde{v} \cdot \nabla H )
+ \beta  \int_{\mathcal{F} _0 }  \tilde{v}  \cdot ( \ell  \cdot\nabla H )
- \beta r \int_{\mathcal{F} _0 }  \tilde{v}  \cdot    (H^\perp - x^\perp  \cdot\nabla H  ) \\
&& -  \beta^2 \int_{\mathcal{F} _0 }  \tilde{v}  \cdot ( H \cdot \nabla H ) \\
& =: & I_1 + I_{2}+ I_{3}+ I_4  + I_5 .
\end{eqnarray*}
Integrating by parts we obtain that  $I_1 = 0$, since $v - \ell -r x^\perp$ is a divergence free vector field tangent to the boundary. Let us stress that there is no contribution at infinity because of the decay properties of the various fields involved, see \eqref{decay}. \par 
On the other hand, using \eqref{EstiHache},  we get that there exists $C>0$ depending only on ${\mathcal{F} _0 }$ such that
\begin{equation*}
| I_2 | + | I_3 | + | I_4 | \leqslant C |\beta| \left( \int_{\mathcal{F} _0 } \tilde{v}^2 \,dx  + \ell^2 + r^2 \right).
\end{equation*}
Let us now turn our attention to $I_5$.
We first  use  that $H$ being curl free, we have 
\begin{equation*}
 \int_{\mathcal{F} _0 }  \tilde{v}  \cdot ( H \cdot \nabla H ) = \frac{1}{2} \int_{\mathcal{F} _0 }  (\tilde{v}  \cdot \nabla  )| H |^2
\end{equation*}
and then an integration by parts and \eqref{Euler13}
to obtain 
\begin{equation*}
 \int_{\mathcal{F} _0 }  \tilde{v}  \cdot ( H \cdot \nabla H ) dx = \frac{1}{2} \int_{\partial \mathcal{S} _0 }  (\tilde{v}  \cdot  n )| H |^2 ds 
  = \frac{1}{2} \ell \cdot \int_{\partial \mathcal{S} _0 }  | H |^2 n ds  +   \frac{1}{2} r   \int_{\partial \mathcal{S} _0 }  | H |^2 x^\perp \cdot n ds .
\end{equation*}
We will make use of the following classical Blasius' lemma (cf. e.g. \cite[Lemma 5]{GLS})
\begin{Lemma}
\label{blasius}
Let $\mathcal{C}$ be a smooth Jordan curve, $f:=(f_1 , f_2)$ and $g:=(g_1 ,g_2 )$ two smooth tangent vector fields on $\mathcal{C}$. Then 
\begin{gather} \label{bla1}
\int_{ \mathcal{C}} (f  \cdot g) n \, ds =  i \left( \int_{ \mathcal{C}} (f_1 - if_2) (g_1 - i g_2) \, dz \right)^*, \\
\label{bla2}
\int_{ \mathcal{C}} (f  \cdot g) (x^{\perp} \cdot n)  \,ds =  \Re \left( \int_{ \mathcal{C}} z (f_1 - if_2) (g_1 - i g_2) \, dz \right).
\end{gather}
where $(\cdot )^*$ denotes the complex conjugation.
\end{Lemma}
Therefore using the Laurent series expansion for $H=(H_1,H_2)$ (see e.g. \cite[Eq. (33)]{GLS}):
\begin{equation*}
H_1(z)-i H_2(z) = \frac{1}{2i \pi z} + {\mathcal O} \left( \frac{1}{z^2} \right),
\end{equation*}
where we identify $(x,y) \in \R^{2}$ and $x+iy \in \C$. Then Cauchy's residue Theorem gives that $I_5 = 0$. \par
Collecting all these estimates it only remains to use Gronwall's lemma to conclude.
\end{proof}
\subsection{Bound of the body acceleration}
The aim of this section is to prove the following  a priori estimate  of the body acceleration.
\begin{Proposition}
\label{Toascoli}
There exists a constant $C>0$ depending only on $\mathcal{S}_0 $, $m $,  ${\mathcal J}$, $\beta$ and  $E(0)$ such that  any classical solution of \eqref{Euler11}--\eqref{Solide1ci} satisfies the estimate
\begin{eqnarray*}
 \| (\ell' , r')  \|_{L^{\infty} (0,T)} \leq C .
\end{eqnarray*}
\end{Proposition}
\begin{proof}
Again, after an approximation procedure, we can use \eqref{SemiW} with, as test functions, the functions $(\Psi_{i} )_{i=1,2,3}$ defined by $\Psi_i  = \nabla \Phi_i $ in $ \mathcal{F}_{0}$ and $\Psi_i = e_i$, for $i=1,2$ and $\Psi_{3 }=  x^\perp $ in $ \mathcal{S}_{0}$.
We observe that 
\begin{equation*}
\big( (\partial_{t} \tilde{v} , \Psi_i   )_\rho \big)_{i=1,2,3}
= \mathcal{M}_1 \begin{pmatrix} \ell_{1} \\ \ell_{2} \\r \end{pmatrix} '(t) + ( A_i   )_{i=1,2,3} (t) ,
\end{equation*}
where
\begin{equation} 
\mathcal{M}_1 := \begin{bmatrix} m \Id_2 & 0 \\ 0 & \mathcal{J} \end{bmatrix}  \ \text{ and } \ 
A_i :=  \int_{  \mathcal{F}_{0}  }  \partial_{t}  v  \cdot  \nabla   \Phi_i  \, dx  .
\end{equation}
Let us examine $A_i$ for $i=1,2,3$. As $\div v =0$ for all time, using \eqref{t1.3sec}-\eqref{t1.4sec} we deduce that
\begin{equation*}
A_i = \int_{  \partial \mathcal{S}_{0}  } \partial_{t}  v  \cdot  n \, \Phi_i \, ds .
\end{equation*}
Now using the boundary condition \eqref{Euler13} and Green's formula we obtain
\begin{equation*}
(A_i )_{i=1,2,3} = \mathcal{M}_2 
\begin{pmatrix} \ell_{1} \\ \ell_{2} \\r \end{pmatrix} ' \text{ where }
\mathcal{M}_2
:= \begin{bmatrix} \displaystyle\int_{\mathcal{F}_0} \nabla \Phi_a \cdot \nabla \Phi_b \, dx \end{bmatrix}_{a,b \in \{1,2,3\}} 
\end{equation*}
encodes the phenomenon of added mass, which, loosely speaking,  measures how much the  surrounding fluid resists the acceleration as the body moves through it. 
We also introduce the matrix 
\begin{equation} \label{InertieMatrix}
\mathcal{M} :=\mathcal{M}_1+\mathcal{M}_2, 
\end{equation}
which is symmetric and positive definite.
The equations  given by \eqref{SemiW}  are now recast as follows:
\begin{equation*}
\mathcal{M}  \begin{bmatrix} \ell \\ r \end{bmatrix} '
 =  \begin{bmatrix}B_i  + C_i   \end{bmatrix}_{i \in \{1,2,3\}} ,
\end{equation*}
where 
\begin{equation*}
B_i := \int_{  \mathcal{F}_{0}  }  v   \cdot    \left[(v-\ell-r x^\perp)\cdot\nabla   \nabla   \Phi_i  \right] \, dx
\text{ and } C_i := - \int_{\mathcal{F}_{0}} r v^\perp \cdot \nabla \Phi_i \, dx.
\end{equation*}
Let us now make use of the decomposition  \eqref{allo}.
We have 
\begin{eqnarray*}
B_i &=&  \int_{\mathcal{F}_{0} } \tilde{v} \cdot \left[ (\tilde{v} \cdot \nabla)  \nabla \Phi_i \right] \, dx
+ \beta  \int_{\mathcal{F}_{0} } H   \cdot  \left[ (\tilde{v} \cdot\nabla) \nabla \Phi_i \right]  \, dx
+ \beta  \int_{\mathcal{F}_{0} }   \tilde{v} \cdot  \left[ (H \cdot\nabla) \nabla \Phi_i \right]  \, dx \\
&& + \beta^2  \int_{\mathcal{F}_{0} }   H   \cdot  \left[  (H  \cdot\nabla)  \nabla   \Phi_i \right]  \, dx
- \int_{  \mathcal{F}_{0}}  \tilde{v} \cdot    \left[((\ell + r x^\perp) \cdot \nabla) \, \nabla \Phi_i \right] \, dx
- \beta \int_{  \mathcal{F}_{0}} H \cdot \left[((\ell + r x^\perp)\cdot\nabla)  \nabla   \Phi_i \right] \, dx , \\
-C_i  &=&  \int_{  \mathcal{F}_{0} } r \tilde{v}^\perp  \cdot  \nabla   \Phi_i \, dx
+ \beta \int_{\mathcal{F}_{0} } r  H^\perp  \cdot  \nabla   \Phi_i \, dx .
\end{eqnarray*}
It then suffices to use 
Proposition  \ref{PE2} and the decaying properties of $  \Phi_i$ and $H$ to obtain a bound of  $\ell'$ and $r'$.

\end{proof}
\subsection{Temporal estimate of the fluid  acceleration}
\label{PF}
We use again  \eqref{SemiW} to get that for any $\xi \in C^\infty_{c} ( \mathcal{F}_{0} ; \R  )$  to get a  bound of 
$\xi  \partial_t v $  in $L^2 ( 0,T; H^{-2} ( \mathcal{F}_{0} ))$. 
\section{End of the proof of Theorem \ref{ThmYudo}}
\label{PL}
We start with an initial data $(\ell_0 , r_0, v_0)$ as in Theorem  \ref{ThmYudo}. 
Then we apply Proposition \ref{PropRegu} to get a sequence of regularized initial data $(\ell^{n}_{0}, r^{n}_{0}, v^{n}_{0}) $. 
According to \cite[Theorem 1]{GS}, these data generate some global-in-time classical solutions that we denote by $(\ell^{n}, r^{n}, v^{n})$. 
These solutions satisfy uniformly in $n$ the estimates given in Section \ref{Sec:APE} since these estimates only depends on the geometry and on some norms of the initial data which are bounded  uniformly in $n$ according to Proposition \ref{PropRegu}. \par
As explained above a classical solution of \eqref{Euler11}--\eqref{Solide1ci} in $[0,T]$ is also a weak solution. Therefore for any test function $\Psi \in \tilde{\mathcal{H} }_{T}$, for any $n$, one has 
\begin{multline} \label{FF}
(  \Psi(T,\cdot ) , \overline{v}^{n}(T,\cdot) )_\rho - (  \Psi (0,\cdot ) , \overline{v}^{n}_0 )_\rho
= \int_0^T \left( \frac{\partial \Psi}{\partial t},\overline{v}^{n} \right)_\rho \, dt
+ \int_0^T \int_{\mathcal{F}_0} v^{n} \cdot \left( v^{n} \cdot \nabla \right) \Psi \, dx \, dt \\ 
- \int_0^T \!\! \int_{\mathcal{F}_0} v^{n} \cdot \left( \left( \ell^{n} + r^{n} x^\perp \right) \cdot \nabla \right) \Psi \, dx\, dt
- \int_0^T \!\! \int_{\mathcal{F}_0} r^{n} (v^{n})^\perp   \cdot \Psi  \, dx\, dt .
\end{multline}
Let us now explain how to pass to the limit. The main step consists in extracting a subsequence to $(\ell^{n}, r^{n}, v^{n})$ by using the above a priori estimates. Using Proposition \ref{Toascoli} and Ascoli's theorem yields that some subsequences of $(\ell^{n})$ and of $(r^{n})$ are converging strongly in $C ([ 0,T])$. Now, we use Propositions \ref{sisi}, \ref{PropRegu} and \ref{PE2} and \eqref{ConsOmega}-\eqref{DefAlpha2} to obtain uniform bounds on $\| \tilde{v}_{n} \|_{L^{2}}$ and $\| \nabla v^{n} \|_{L^p (\mathcal{F}_{0})}$ in $L^\infty ( 0,T)$.
Using the bound obtained in Section \ref{PF} and the Aubin-Lions lemma we deduce that some subsequence of $(v^{n})$ converges strongly in $C([0,T]; L^2_{loc})$. \par
Now it is elementary to check that these convergence are sufficient to pass to the limit in \eqref{FF}.
This achieves the proof of Theorem \ref{ThmYudo}. \par
\ \par \noindent
{\bf Acknowledgements.} The  authors are partially supported by the Agence Nationale de la Recherche, Project CISIFS, Grant ANR-09-BLAN-0213-02. \par
\end{document}